\documentclass  {amsart}

\usepackage{amsmath, amssymb, amsfonts, latexsym}

\usepackage{amsfonts}
\usepackage{amsmath}

\usepackage{amssymb}

\theoremstyle{plain}
\newtheorem{acknowledgement}{Acknowledgement}

\def\R{{\Bbb R}} 

  \def\laplacian{\bigtriangleup}
\def\loc{\text{loc}}  
\def\intave#1{-\kern-10.7pt\int_{\,#1}}

\def\D{\nabla}
\def\g{\gamma}
\def\wc{\rightharpoonup}
\def\<{\langle}

\def\({\left(}
\def\){\right)}

\def\intave#1{-\kern-10.7pt\int_{\,#1}}
\def \cal{\mathcal}

\begin{document}
\title{A New Approximation of   Relaxed Energies for
 Harmonic Maps and   the Faddeev Model}

\numberwithin{equation}{section}
\author{Mariano Giaquinta, Min-Chun Hong and Hao Yin}
\address{Mariano Giaquinta, Scuola Normale Superiore,
56126 Pisa, Italy}  \email{m.giaquinta@sns.it}

\address{Min-Chun Hong, Department of Mathematics, The University of Queensland\\
Brisbane, QLD 4072, Australia}  \email{hong@maths.uq.edu.au}

\address{Hao Yin, Department of Mathematics, The University of Queensland\\
Brisbane, QLD 4072, Australia and Department of Mathematics,
Shanghai Jiaotong University, Shanghai, China}
\email{yinhao78@gmail.com}

\begin{abstract}
 {We propose a new approximation
for  the relaxed energy $E$ of the Dirichlet energy and prove that
 the minimizers of the approximating functionals converge to a
 minimizer $u$ of the relaxed energy, and that $u$ is  partially regular
without using the concept of Cartesian currents.
 We also use the same approximation method to study
 the variational problem of the relaxed energy for the
Faddeev model and prove the existence of minimizers  for the relaxed
energy $\tilde{E}_F$ in  the class of maps with Hopf degree $\pm
1$. }\end{abstract}

 \maketitle

\pagestyle{myheadings} \markright {} \markleft { }

\vfuzz2pt 
\hfuzz2pt 

\newtheorem{thm}{Theorem}[section]
\newtheorem{cor}[thm]{Corollary}
\newtheorem{lem}[thm]{Lemma}
\newtheorem{prop}[thm]{Proposition}
\newtheorem{defn}[thm]{Definition}
\newtheorem{rem}[thm]{Remark}

\newcommand{\norm}[1]{\left\Vert#1\right\Vert}
\newcommand{\nnorm}[1]{\interleave #1 \interleave}
\newcommand{\abs}[1]{\left\vert#1\right\vert}
\newcommand{\set}[1]{\left\{#1\right\}}
\newcommand{\Real}{\mathbb R}
\newcommand{\eps}{\varepsilon}
\newcommand{\To}{\longrightarrow}
\newcommand{\BX}{\mathbf{B}(X)}
\newcommand{\A}{\mathcal{A}}
\newcommand{\pfrac}[2]{\frac{\partial #1}{\partial #2}}

\section{Introduction}

Let $\Omega\subset\R^3$ be an open bounded domain with smooth
boundary $\partial \Omega$. For each $p>0$, set
$$W^{1,p}(\Omega ; S^2) =\left \{ u\in W^{1,p}(\Omega ; \R^3): \,\, |u|=1\text { a.e. on } \Omega .\right \}
$$
For a map $u\in W^{1, 2}(\Omega ; S^2)$, we consider
 \begin{equation}
     E(u, \Omega )=\int_{\Omega}|\nabla u|^2\,dx.  \end{equation}
A map $u\in W^{1,2}(\Omega ; S^2)$ is said to be a harmonic map if
$u$ satisfies
$$ \laplacian u+|\nabla u|^2 u=0 $$
in the weak sense.

In \cite {HL}, Hardt-Lin discovered a gap phenomenon for the
Dirichlet energy; i.e., there is a given smooth boundary value $\g:
\partial\Omega\to S^2$ with $\text {deg } \g =0$ such that
$$\min_{u\in W^{1,2}_{\gamma}(\Omega , S^2)}\int_{\Omega}|\D u|^2\,dx <\inf_{v\in W^{1,2}_{\g}\cap
C^{\infty}(\bar\Omega , S^2)}\int_{\Omega} |\D u|^2\,dx. $$ It is
a very interesting problem whether the above right-hand term
$$\inf_{v\in H^1_{\g}\cap C^{\infty}(\bar \Omega , S^2)}\int_{\Omega} |\D
u|^2\,dx$$ can be attained by a map  $u\in  W^{1,2}_{\g}\cap
C^{\infty}(\bar \Omega , S^2)$ or not. For $u\in  W^{1,2}(\Omega ; S^2)$   we consider  the vector field
 $D(u)$ given by
$$D(u):=(u\cdot u_{x_2}\times u_{x_3},u\cdot u_{x_3}\times u_{x_1},
u\cdot u_{x_1}\times u_{x_2}). $$
Given
$u_0\in C^{\infty}_{\g}(\bar \Omega ; S^2)$ for a
map $u\in  W^{1,2}_{\g}(\bar \Omega ; S^2)$, we set
 $$ L(u):= L(u,u_0)=\frac 1 {4\pi}\sup_{\xi :\Omega\to \R, \|\D\xi\|_{L^{\infty}}\leq 1}
\int_{\Omega} [D(u)-D(u_0)]\cdot \D\xi\,dx .
$$
The relaxed energy functional for the Dirichlet energy $E(u)$ is then
defined by
$$F(u)=\int_{\Omega}|\D u|^2\,dx+8\pi L(u). $$
Bethuel-Brezis-Coron in \cite {BBC} proved that $F$ is
sequentially lower semi-continuous and satisfies
 \begin{equation} \label{1.2}\inf_{u\in W^{1,2}_{\g}\cap
C^{\infty} (\bar \Omega , S^2)}\int_{\Omega} |\D
u|^2\,dx=\min_{u\in W^{1,2}_{\gamma}(\Omega ,S^2)}  F(u). \end{equation}
Moreover, each minimizer of $F$ in $W^{1,2}_{\g}(\Omega
; S^2)$ is also a weak harmonic map.

There is an interesting problem whether a minimizer of the relaxed
energy $F$ in (\ref {1.2}) is regular.  Giaquinta, Modica and
Soucek in \cite {GMS1} proved that a minimizer $u$ of the relaxed
energy $F$ in $W^{1,2}_{\gamma}(\Omega ; S^2)$ is smooth in a set
$\Omega_0\subset  \Omega$ and ${\cal
H}^{1}(\Omega\backslash\Omega_0)<\infty$, where ${\cal H}^{1}$ is
the Hausdorff measure. Bethuel and Brezis in \cite {BB} proved
that minimizers for a modified relax energy $F_{\lambda}(u)
=\int_{\Omega}|\D u|^2\,dx+\lambda 8\pi L(u)$ with $0\leq \lambda
<1$ in $W^{1,2}_{\gamma}(\Omega ; S^2)$ have  at most isolated
singularities.

In the first part of this paper, we propose a new approximation
for  the relaxed energy $F$ of the Dirichlet energy. More
precisely, for a parameter $\varepsilon$, we consider the family of
functionals
  \begin{equation}  E_{\varepsilon} (u;\Omega )=\int_{\Omega}\left (|\nabla u|^2+\varepsilon^2 |\nabla u |^4\right
)\,dx
  \end{equation}
for  maps $u \in W^{1,4}$. Of course, there is a minimizer $u_{\varepsilon}$
of $E_{\varepsilon}$ in $W^{1,4}_{\gamma}(\Omega ; S^2)$; i.e.,
$$E_{\varepsilon} (u_{\varepsilon} ;\Omega )\leq E_{\varepsilon} (v;\Omega
);\quad \forall v\in   W^{1,4}_{\gamma}(\Omega ; S^2)
$$
and it can be proved that each $u_{\varepsilon}$ is smooth. We shall prove the following.

\begin{thm}
    \label{thm:approx}
 For $\varepsilon >0$, let
$u_{\varepsilon}$ be a minimizer of $E_{\varepsilon}$ in
$W^{1,4}_{\gamma}(\Omega ; S^2)$. As $\varepsilon_k\to 0$,
minimizers $u_{\varepsilon_k}$ weakly converge  (possible passing
subsequence) to a minimizer $u$ of the relaxed energy $F$ in
$W^{1,2}_{\gamma}(\Omega ; S^2)$. Moreover, $u$ is harmonic and
smooth in an open subset $\Omega_0\subset \Omega$ with $\cal
H^1_{\loc}(\Omega\backslash \Omega_0)<+\infty $.
\end{thm}
During the proof of Theorem 1.1, we show that the limit map $u$ is
partial regular  by a new approach, which is different from the
one in \cite {GMS1}, and does not use the concept of Cartesian
currents.

Giaquinta, Modica and Soucek  (\cite {GMS1},\cite {GMS2})  viewed
the relaxed energy functional in terms of Cartesian currents. Let $\g$ be a
smooth function defined in a neighborhood $\tilde\Omega$ of
$\bar\Omega$ with values in $S^2$. Set
$$\text{cart}_{\g}^{2,1}(\tilde \Omega\times
S^2):=\{T\in\text{cart}^{2,1}(\tilde \Omega\times S^2)\, |\,\,
(T-G_{\g})\llcorner (\tilde \Omega\backslash \bar\Omega)=0\}.
$$
It was shown in \cite {GMS1} that for any $u\in W^{1,2}_{\g}
(\tilde \Omega ,S^2)$, there is a $1$-dimensional integer rectifiable
current $L_u$  of minimal mass among the integer rectifiable
$1$-dimensional currents $L$ with $\text{spt} L_u\subset \bar
\Omega$ such that $-\partial L=P(u)$ and
$${\bf M} (L_u)=L(u), $$
where ${\bf M}(L_u)$ is the mass of the current $L_u$. Then
$$\cal D (T_u; \tilde \Omega )= \int_{\tilde \Omega }|\nabla
u|^2\,dx +8\pi {\bf M}(L_u)=F(u;\tilde \Omega ).
$$
For each $L=\tau (\cal L,\theta, \vec{L})$, we denote by $e(T)$
the energy density of the current $T:=G_u+L\times [[S^2]]$; i.e.,
$$e(T):=|\nabla u|^2\,dx + 8\pi \theta \cal H^1\llcorner \cal
L.
$$
 Then, in the spirit of \cite{GMS2} we improve Theorem 1.1 as:
\begin{thm}
    \label{thm:approx1}
 For a parameter $\varepsilon >0$, let
$u_{\varepsilon}$ be a minimizer of $E_{\varepsilon}$ in
$H_{\gamma}^{1,4}(\Omega ,S^2)$ as in Theorem 1.1. Then, there is
a sequence $\varepsilon_k\to 0$ such that
\[ G_{u_{\varepsilon_k}}\wc T_u=G_{u} +L_u\times [[S^2]] \] weakly in
$\text{cart}_{\g}^{2,1} (\tilde \Omega\times S^2)$ with $L_u=\tau
(\cal L,\theta, \vec{L} )$, $G_u$ denotes the  graph current of
$u$ in $\Omega$,
\[\cal
D(u_{\varepsilon_k} , \tilde \Omega ) \to D(T_u,\tilde \Omega
\times S^2)\] and
\[|\nabla u_{\varepsilon_k}|^2\,dx\wc |\nabla u|^2\,dx +8\pi \theta
\cal H^1\llcorner  \cal L \] in the sense of  Radon measures.
Moreover,  $u_{\varepsilon_k} \to u$ strongly in
$W^{1,2}(\Omega_0, S^2)$, where $\Omega_0$ is the open set in
Theorem 1.1.
\end{thm}

Furthermore, Theorem 1.1 can be generalized to the case of maps
between two Riemannian manifolds. Let $\cal M$ be a  Riemannian
manifold with boundary $\partial \cal M$ and $\cal N$ another
compact Riemannian manifold without boundary. Let $\g$ be a smooth
map from $\partial\cal M$ to $\cal N$ which can be extended to a
smooth map $u_0$ from $\cal M$ to $\cal N$. For a map $u:\cal M\to
\cal N$, we consider the Dirichlet energy
 $E_{\cal M}(u)=\int_{\cal M} |\nabla u|^2\,d \cal M$. Then, we can define the relaxed
 energy  for  the Dirichlet energy $E_{\cal M}$ by
 $$\tilde E_{\cal M}(u):=\inf\left\{\liminf_{k\to \infty} E_{\cal M}(u_k)\,\left |\,
  \{u_k\}\subset  C^{\infty}_{\g}(\cal M,\cal N), u_k\wc u \text { weakly in }W^{1,2}(\cal M,\cal N)\right .\right \}
 $$
Under certain topological conditions on the manifold $\cal N$,
Giaquinta-Mucci \cite {GM} found a representation formula of
$\tilde E_{\cal M}(u)$. Without using the explicit formula $\tilde
E_{\cal M}(u)$ or any topological assumption on $\cal N$, we have

\begin{thm}
    \label{thm:manifold}
 For a parameter $\varepsilon >0$, let
$u_{\varepsilon}$ be a minimizer of $$E_{\varepsilon}(u; \cal M)
=\int_{\cal M} |\nabla u|^2 +\varepsilon^2 |\nabla u|^{n+1}\,d\cal
M$$ in $W_{\gamma}^{1,n+1}(\cal M, \cal N)$. As $\varepsilon_k\to
0$, minimizers $u_{\varepsilon_k}$ weakly converge  (possible
passing subsequence) to a minimizer $u$ of the relaxed energy
$\tilde E_{\cal M}$ in $W_{\gamma}^{1,2}(\cal M,\cal N)$.
Moreover, assuming that $\cal N$ is a homogenous manifold, $u$ is
harmonic and  smooth in an open subset $\cal M_0\subset   \cal M$
with $\cal H^{n-2}_{\loc}( \cal M\backslash  \cal M_0)<+\infty$.
\end{thm}

After we finished an early version of our paper, Fanghua Lin
informed us that a result related to Theorem \ref{thm:manifold}
was studied by him in \cite {L}, but his proof is different from
ours.

\medskip
 In the second part of this paper, we apply the same
approximation method to the study of
 the relaxed energy for the Faddeev model. The Faddeev theory has created a lot of interest in physics as a
successful effective field theory (\cite {F2}).  The underlying
idea is very old for, in the 19th century, Lord Kelvin proposed
that atoms could be described as knotted vortex tubes in the
ether. Modern physics replaces the ether by `fields' of which the
classical gravitational and electromagnetic fields are the  most
studied. It is known from work in condensed matter theory that
field theories (such as sigma models) are useful in understanding
the behaviour of composite particles. Thus  Faddeev  \cite{model}
in 1979 proposed that closed, knotted vortices could be
constructed in a definite dynamical system; namely the Faddeev
model. The study of the dynamics of knots as the solution
configurations of suitable Lagrangian field-theory equations has
only  been brought to light by Faddeev and Niemi in \cite {FN}.
They employed powerful numerical algorithms to show that a
ring-shaped charge one soliton exists. In mathematics,  the
Faddeev energy  for a map $u$  from $\Real^3$ to $S^2$ is given by
\begin{equation*}
    E_F(u)=\int_{\Real^3}[\abs{\nabla u}^2+\frac 1 2 \sum_{1\leq k<l\leq 3}\abs{\partial _k u\times \partial_l u}^2 ]dx.
\end{equation*}
For any interesting configuration in physics, it is required that
$u$ converges to a constant sufficiently fast near the infinity.
Under this assumption, we may compactify $\Real^3$ by adding a
point representing the infinity and view  $u$  to be a map from
$S^3$ to $S^2$. Hence any smooth field configuration $u$ can be
characterized by the topological charge $Q(u)$ given by the Hopf
degree of $u$ from $\R^3$ to $S^2$. It is well known that the Hopf
degree can be expressed analytically as
\begin{equation*}
    Q(u)=\frac{1}{16\pi^2}\int_{\Real^3}\eta\wedge
    u^*(\omega_{S^2}),
\end{equation*}
where $\omega_{S^2}$ is a volume form on $S^2$ and
$d\eta=u^*\omega_{S^2}$. Such an $\eta$ exists because
$u^*\omega_{S^2}$ is closed. Moreover, one can see  that $Q(u)$ is
independent of the choice of $\eta$.
In this paper, unless stated otherwise, we will always take
\begin{equation}\label{def:eta}
    \eta=\delta(-\frac{1}{4\pi\abs{x}}\star u^*\omega_{S^2}).
\end{equation}
Here $\star$ means convolution and $\delta$ is the formal adjoint of $d$.

Vakulenko and Kapitanski in \cite{VK} found a
lower-bound for the Faddeev energy; i.e. there is a constant $C>0$
such that
\begin{equation*}
    E_F(u)\geq C Q(u)^{3/4}.
\end{equation*}

A minimizer of $E_F$ among all maps with the same Hopf degree is
called a hopfion, or a Hopf soliton. For each $d\in \mathbb Z$,
Faddeev \cite {F2} suggested that there is a knotted  minimizer of
$E_F$ in $H_d$, where $H_d$ is the class of smooth maps
$u:\Real^3\to S^2$ which have bounded energy and approach a
constant value sufficiently fast as required by physics  at the
infinity (for a more precise definition see Section 3). Many
contributions have been made by physicists (\cite {RV}). A natural
functional space for the minimization of Faddeev energy is
\begin{equation*}
    X=\{u:\Real^3\to S^2\,|  \quad E_F(u)< +\infty\}.
\end{equation*}
Lin and Yang (\cite{LY1}, \cite{LY2}) showed that for infinitely many integers $d$'s, including
the case $d=\pm 1$, there is a minimizer of $E_F$ in $X_d$, where
\begin{equation*}
    X_d=\{u\in X\,| \quad Q(u)=d\}.
\end{equation*}
 (See \cite{Hang} for
 further development). However, it is still unknown whether the minimizers obtained by Lin-Yang \cite {LY1}
 are
smooth or not. Therefore, there might occur a gap phenomenon for
the Faddeev energy similarly to one for harmonic maps; i.e.
\begin{equation}\label{1.4}
    \inf_{u\in X_d} E_F(u)<\inf_{u\in H_d} E_F(u).
\end{equation}
 A similar situation occurs in the Skyrme model as pointed out by Esteban and M\"uller \cite{EM}.

It is a very challenging problem whether the infimum of  Faddeev
energy on the righthand of (\ref{1.4}) can be achieved or not.
We consider the relaxed energy for
$u\in X_d$,
\begin{equation*}
    \tilde{E}_F(u)=\inf\left \{
    \liminf_{k\to \infty} E_F(u_k)\,| \quad \{u_k\} \subset H_d, u_k\rightharpoonup u \mbox{ weakly
    }
    \right\},
\end{equation*}
where `weakly converging' means in the  sense of bounded Faddeev
    energy. For a map $u\in X_d$ without any sequence
$u_k$ converging weakly to $u$,  we take the value
$\tilde{E}_F(u)$ to be $+\infty$.

However, due to the complexity of the Faddeev energy,  this time
we do not have an explicit formula as, for instance, in \cite{BBC}
and \cite {GMS2}. By using the same approximation method as
previously, we will find a minimizer for the relaxed energy
$\tilde{E}_F$ in some cases. More precisely, we have
\begin{thm}
    For $d=\pm 1$, there is a minimizer of the relaxed energy $\tilde{E}_F$ in
    $X_d$. Moreover,
     \begin{equation}\inf_{u\in H_{\pm 1}} E_F(u) =\min_{u\in X_{\pm
     1}}\tilde{E}_F(u).\end{equation}
\end{thm}
Our proofs are based on the ideas of
Lin-Yang in \cite {LY1}. We will introduce a perturbed energy ${E}_{F,\varepsilon}$ with a
parameter $\varepsilon$.  It turns out that the analysis involved
in the perturbed variational problem  for $\varepsilon>0$ is much
easier to understand. Modifying an idea of  Ward in \cite{Ward}
and Lin-Yang in \cite{LY2}, we show that the new minimizing problem
has a solution for $d=\pm 1$.
 Then, letting $\varepsilon$ go to zero,  we prove that the minimizers of the perturbed
 energy ${E}_{F,\varepsilon}$ in $X_{\pm 1}$
 will converge to a minimizer of the relax energy $\tilde{E}_F$ in $X_{\pm 1}$.

\section{Relaxed energy for harmonic maps}

\begin{lem}
Let $u_{\varepsilon}$ be a minimizer of the functional
$E_{\varepsilon}$. Then, for all $0<\rho \leq R$ with
$B_R(x_0)\subset \Omega$, we have
  \begin{align*}
&\quad R^{-1}\int_{B_R(x_0)} \left [ |\nabla u_{\varepsilon} |^2
+\varepsilon^2 |\nabla u_{\varepsilon}|^{4}\right
]\,dx-\rho^{-1}\int_{B_{ \rho}(x_0)} \left [
|\nabla u_{\varepsilon} |^2 +\varepsilon^2 |\nabla u_{\varepsilon}|^{4}\right ]\,dx\\
&=\int_{B_R\backslash B_{\rho}(x_0)}\left [ 1+2\varepsilon^2
|\nabla u_{\varepsilon}|^{2}\right ] |\partial_r
u_{\varepsilon}|^2 r^{-1} \,dx-\int_{\rho}^R \int_{B_r(x_0)}
2\varepsilon^2 |\nabla u_{\varepsilon}|^{4}(y)\,dy r^{-2}\,dr.
 \end{align*}
 \end{lem}
\begin{proof} Without loss of generality, we assume $x_0=0$.
Let $\phi (x)=(\phi^1 (x), ... , \phi^n (x))\in
C^1(\Omega ;\R^n)$ be a vector field having compact support in
$\Omega$. For a function $\phi_t (x)=x+t\phi (x)$ with
$x=(x^1,...,x^n)$ and for a function $u(x)$, set
$u^{t,\phi}(x):=u_{\varepsilon}(x+t\phi (x))$.

We see $\nabla_{x^i} u^{t,\phi}(x)=u_{x^k}(x+t\phi
)[\delta_{ik}+t\phi^k_{x^i}]$. Since $u_{\varepsilon}$ is a
minimizer of $E_{\varepsilon}$, it follows from $\frac d {dt}
E_{\varepsilon} (u^{t,\phi}, \Omega ) |_{t=0}=0$ that
  \begin{align*}
\int_{\Omega } &\left (|\nabla u_{\varepsilon}|^2 +\varepsilon^2
|\nabla u_{\varepsilon}|^{4} \right )\text {div} \phi -2
{u_{\varepsilon}}_{x^i}{u_{\varepsilon}}_{x^k}\phi^k_{x^i}
(1+2\varepsilon^2 |\nabla u_{\varepsilon}|^2 )\,dx=0.
 \end{align*}
For a given  ball $B_r(0)\subset \Omega$,  we choose $\phi (x)
=\xi (|x|)x$ with

\begin{displaymath}
\xi (s)=\left\{\begin{array}{ll}1&\textrm{ for $s\leq r$}\\
1+\frac {r-s}h &\textrm{ for $r\leq s\leq r+h$}\\
0 &     \textrm{ for $s\geq r+h$}.
\end{array}\right.
\end{displaymath}

Letting $h\to 0$, we obtain that for almost every $r$
 \begin{align*}
&\int_{B_r} [|\nabla u_{\varepsilon}|^2 +\varepsilon^2 |\nabla u_{\varepsilon}|^{4}]\,dx-r\int_{\partial B_r} [|\nabla u_{\varepsilon}|^2 +\varepsilon^2
|\nabla u_{\varepsilon}|^{4}]\,d H^{n-1}\\
&=-\frac 2 {r}\int_{\partial B_r} (1+2\varepsilon^2 |\nabla
u_{\varepsilon}|^{2}) |x^i\partial_{x^i}
u_{\varepsilon}|^{2}\,dH^{n-1} +\int_{B_r} 2\varepsilon^2 |\nabla
u_{\varepsilon}|^{4}\,dx .\end{align*}

 Multiplying by $r^{-2}$ both sides
of the above identity and integrating with respect to $r$ from
$\rho$ to $R$ yields the result. \end{proof}

We now complete the proof of Theorem \ref{thm:approx} and Theorem \ref{thm:approx1}.

\begin{proof}[Proof of Theorem \ref{thm:approx}]

 Let $u_{\varepsilon}$ be a minimizer of the functional
$E_{\varepsilon}$ in $W_{\g}^{1,4}(\Omega ;S^2)$ and $u$ the weak
limit of the sequence in $W^{1,2}$. Since $F$ is lower
semi-continuous, we have
\begin{align}
& F(u; \Omega) \leq  \liminf_{\varepsilon_k\to 0} \int_{\Omega}
|\nabla u_{\varepsilon_k}|^2\,dx \leq \liminf_{\varepsilon_k\to 0}
\int_{\Omega} |\nabla u_{\varepsilon_k}|^2+\varepsilon_k^2 |\nabla
u_{\varepsilon_k}|^4
 \,dx\\
 & \leq  \inf_{v\in C_{\g}^{\infty}(\Omega ;S^2)}  \liminf_{\varepsilon_k\to 0} \int_{\Omega} |\nabla v|^2+\varepsilon_k^2 |\nabla v|^4
 \,dx = \inf_{v\in C_{\g}^{\infty}(\Omega ;S^2)} \int_{\Omega} |\nabla v|^2\,dx.
\end{align}
Using (1.2), we note
\[ \inf_{v\in C_{\g}^{\infty}(\Omega ;S^2)} \int_{\Omega} |\nabla v|^2\,dx =\min_{v\in W^{1,2}_{\g}(\Omega ;S^2)}
F(v).\] Then it follows from (2.1) and (2.2) that $u$ is a
minimizer of the relaxed energy $F$ in $W^{1,2}(\Omega ;S^2)$.
Moreover, we have
  \begin{equation}
  \lim_{\varepsilon_k\to 0} \int_{\Omega}\varepsilon_k^2 |\nabla u_{\varepsilon_k}|^4
 \,dx =0. \end{equation}

 We define
 $$\Sigma =\bigcap_{R>0} \left\{x_0\in \Omega : B_R(x_0)\subset \Omega,\quad \liminf_{\varepsilon_k\to 0}
  \frac 1 R\int_{B_R(x_0)} |\nabla u_{\varepsilon} |^2 \,dx\geq \varepsilon_0\right
 \}
 $$
 for a sufficiently small constant $\varepsilon_0$ to be fixed later.

As in \cite{Sc}, we can show that $\Sigma$ is
relatively closed inside $\Omega$. In fact, let $x_j$ be a
sequence in $\Sigma$ and $x_j$ goes to $x\in \Omega$. To see that
$x\in \Sigma$, we need to show for all $R>0$,
\begin{equation*}
    \liminf_{\varepsilon\to 0} \frac{1}{R}\int_{B_R(x)} \abs{\nabla u_\varepsilon}^2 dx\geq \varepsilon_0.
\end{equation*}
Let $r<R$ (but very close). Since $x_j\in \Sigma$,
\begin{equation*}
    \liminf_{\varepsilon\to 0}\frac{1}{r}\int_{B_r(x_j)}\abs{\nabla u_\varepsilon}^2 dx\geq \varepsilon_0.
\end{equation*}
Choose some large $j$ so that $B_r(x_j)\subset B_R(x)$, then
\begin{equation*}
    \frac{1}{r}\int_{B_r(x_j)}\abs{\nabla u_\varepsilon}^2 dx\leq \frac{R}{r} \frac{1}{R}\int_{B_R(x)}\abs{\nabla u_\varepsilon}^2 dx.
\end{equation*}
Hence,
\begin{equation*}
    \frac{R}{r}\liminf_{\varepsilon_k\to 0}\frac{1}{R}\int_{B_R(x)} \abs{\nabla u_\varepsilon}^2dx\geq \varepsilon_0.
\end{equation*}
Since $r$ can be chosen very close to $R$,
\begin{equation*}
    \liminf_{\varepsilon_k\to 0}\frac{1}{R}\int_{B_R(x)} \abs{\nabla u_\varepsilon}^2dx\geq \varepsilon_0.
\end{equation*}
This is true for all $R>0$. Therefore, $x\in \Sigma$ and $\Sigma$
is closed. The same argument of Schoen \cite {Sc} yields
 $\cal H_{\loc}^1(\Sigma )<+\infty$.

 A consequence is that for each $x_0\notin \Sigma$, there are a sequence $\varepsilon_k\to 0$ and   a ball
 $B_{R_0}(x_0)\subset\Omega\backslash\Sigma$ such that
\[\lim_{\varepsilon_k\to 0} R_0^{-1}\int_{B_{R_0}(x_0)} |\nabla
u_{\varepsilon_k} |^2 \,dx<\varepsilon_0.\]

Now, we prove that $u$ is smooth around such points $x_0$. In fact, for each
$y\in B_{R_0/2}(x_0)$, trivially
\begin{equation*}
    \lim_{\varepsilon_k\to 0}\frac{2}{R_0}\int_{B_{R_0/2}(y)}\abs{\nabla u_{\varepsilon_k}}^2 dx
    \leq \lim_{\varepsilon_k\to 0}2
    \frac{1}{R}\int_{B_{R_0}(x_0)}\abs{\nabla u_{\varepsilon_k}}^2 dx< 2\varepsilon_0,
\end{equation*}
while, applying (2.3)  to Lemma 2.1,  for any $0<\rho <R_0/2$, we have
\begin{align*}
&    \lim_{\varepsilon_k\to 0} \frac 2 {R_0}\int_{B_{R_0/2}(y)}
|\nabla u_{\varepsilon_k} |^2 \,dx- \lim_{\varepsilon_k\to 0}
\rho^{-1}\int_{B_{ \rho}(y)}
|\nabla u_{\varepsilon_k} |^2  \,dx\\
&=\lim_{\varepsilon_k\to 0}  \int_{B_{R_0/2}\backslash
B_{\rho}(y)}\left [ |\partial_r
u_{\varepsilon_k}|^2+\varepsilon_k^2 |\nabla
u_{\varepsilon_k}|^{2} |\partial_r u_{\varepsilon_k}|^2\right ]
r^{-1} \,dx\geq 0.
 \end{align*}
 Therefore,  for each $y\in B_{R_0/2}(x_0)$ and for each $\rho\in
(0,R_0/2)$, we have
 \begin{align}
 &\frac 1 {\rho}\int_{B_{\rho }(y)}|\nabla u|^2\,dx\leq
 \lim_{\varepsilon_k\to 0}
\rho^{-1}\int_{B_{ \rho}(y)} |\nabla u_{\varepsilon_k} |^2  \,dx
 <2 \varepsilon_0
  \end{align}
 for a sufficiently small constant $\varepsilon_0>0$.

 Since $u$ is a weak harmonic map with the property (2.4),
it follows, similarly to the proof in \cite {E} (see also Lemma
3.3.13 of \cite {LW}) for stationary harmonic maps, that $u$ is
smooth inside $B_{R_0/2}(x_0)$. This proves Theorem
\ref{thm:approx}.
\end{proof}

\begin{proof}[Proof of Theorem \ref{thm:approx1}]
 Let $T$ belong to $\text{cart}^{2,1} (\Omega \times S^2)$, i.e.
$$T=G_{u_T}+L_T\times [[S^2]].
$$
The Dirichlet integral of $T$ is given by
$$D(T;\Omega\times S^2)= \int_{\Omega } |\nabla
u_T|^2\,dx +8\pi M(L_T).
$$
As $\varepsilon_k\to 0$, passing to a subsequence  we have
\[e(G_{u_{\varepsilon_k }})\wc \mu_0\]
 for a measure $\mu_0$.
 For any $\psi\in C_c^0(\tilde \Omega )$ with  $\psi \geq 0$, we
 consider the functional in $\text{cart}_{\g}^{2,1}(\tilde \Omega
 \times S^2)$
 \[ \cal E(T):=\int \psi (x) e(T).\]
 We know that $\cal E$ is lower semi-continuous with respect to the
 weak convergence in $\text{cart}^{2,1}$ which implies
 \[ e (T)(\psi )\leq \liminf_{k\to\infty} e (G_{u_{\varepsilon_k }})(\psi )\leq \mu_0(\psi ).\]
 This means $e(T)\leq \mu_0$. By Theorem \ref{thm:approx}, we have
 $\cal D(u_k,\Omega )\to \cal D (T; \Omega \times S^2)$. Then, we
 have
 \[ \mu_0(\Omega )=e(T)(\Omega ) . \]
 This shows  $\mu_0=e(T)$ as we claimed.

 On the other hand, if $U$ be a subset inside $\Omega\backslash \Omega_0$, according to \cite{GMS2} p. 436, we can prove
$L_T\llcorner U=0$.
 Therefore, we have
 \[T_u|_{U}=G_u|_{U}.\]
 Therefore
 $$\int_{U}|\nabla
 u_{\varepsilon_k}|^2\,dx \to \int_{U}|\nabla
 u_{\varepsilon_k}|^2\,dx,
 $$
 and this implies $u_k\to u$ strongly in $W^{1,2}_{\loc}(\Omega
 \backslash \bar\Omega_0 ; S^2)$.
\end{proof}

\begin{proof}[Proof of Theorem \ref{thm:manifold}]

It is well-known that $C^{\infty}_{\g}(\cal M, \cal N)$ is dense
in $W^{1,n+1}_{\g}(\cal M, \cal N)$.  Then
 \begin{align*}
 &\tilde E_{\cal M}(u) =\inf\left\{\liminf_{k\to \infty} E_{\cal M}(u_k)\,\left |\,
  \{u_k\}\subset  C^{\infty}_{\g} (\cal M,\cal N), u_k\wc u \text { weakly in } W^{1,2} (\cal M,\cal N)\right .\right
  \}\\
  &=\inf\left\{\liminf_{k\to \infty} E_{\cal M}(u_k)\,\left |\,
  \{u_k\}\subset  W^{1,n+1}_{\g}(\cal M,\cal N), u_k\wc u \text { weakly in }W^{1,2}(\cal M,\cal N)\right .\right
  \}
\end{align*}

Let $u_{\varepsilon}$ be a minimizer of the functional
$E_{\varepsilon}(\cdot ;\cal M)$ in $W^{1,n+1}_{\g}(\cal M, \cal
N)$.  Passing to a subsequence, $\nabla_k u_\varepsilon$ converges to
$\nabla_k u$ weakly in $L^2$. Then, we have
\begin{align}
& {\tilde E}_{\cal M}(u) \leq  \liminf_{\varepsilon_k\to 0}
E_{\cal M} (u_{\varepsilon_k}) \leq \liminf_{\varepsilon_k\to 0}
E_{\varepsilon_k} (u_{\varepsilon_k}; \cal M) \\
& \leq \inf_{v\in W^{1,n+1}_{\g}(\cal M,\cal N)} E_{\cal M}
(v)=\inf_{v\in
 C^{\infty}_{\g}(\cal M,\cal N)} E_{\cal M} (v).
\end{align}
From (2.5)-(2.6), $u$ is a minimizer of the relaxed energy
${\tilde E}_{\cal M}$ in $H^{1,2}$ and moreover, we have
  \begin{equation}
  \lim_{\varepsilon_k\to 0} \int_{\cal M}\varepsilon_k^2 |\nabla
  u_{\varepsilon_k}|^{1+n}
 \,d\cal M =0. \end{equation}

If $\cal N$ is  a homogeneous manifold, we claim that $u$ is a
weak harmonic map.

Let $X_i$ be the Killing field on $\cal N$ as in H\'elein \cite{Helein}. Consider the
vector field
\begin{equation*}
    \xi_i =\langle X_i, \nabla u_\varepsilon\rangle +2 \varepsilon^2 \abs{\nabla u_\varepsilon}^{n-1}\langle X_i ,\nabla u_\varepsilon\rangle.
\end{equation*}

We claim that $\mbox{div} \xi_i=0$.

To see this, let $\varphi$ be some cut-off function compactly
supported in $\Omega$. Since $u_\varepsilon$ is the minimizer of $E_{\varepsilon}$, we can use $\varphi X_i(u)$ as a testing vector
field in the Euler-Lagrange equation to get
\begin{equation*}
    \int_{\cal M} \langle \nabla_k(\varphi X_i(u_\varepsilon)),
    \nabla_k u_\varepsilon +2\varepsilon^2 \abs{\nabla u_\varepsilon}^{n-1} \nabla_k u_\varepsilon\rangle d\mu=0.
\end{equation*}
Since $X_i$ is a Killing vector field,
\begin{equation}\label{eqn:killing}
    \int_{\cal M} \nabla_k \varphi\langle X_i(u_\varepsilon), \nabla_k u_\varepsilon +2\varepsilon^2
    \abs{\nabla u_\varepsilon}^{n-1} \nabla_k u_\varepsilon\rangle d\mu=0.
\end{equation}
Therefore, $\mbox{div} \xi=0$ in distribution sense.

Since $u_{\varepsilon_k}$ converges to $u$ strongly in $L^2$ and
$X_i$ are smooth vector fields on $\cal N$, $X_i(u_{\varepsilon_k})$
converges to $X_i(u)$ strongly in $L^2$. Letting $\varepsilon_k$
go to zero in equation (\ref{eqn:killing}) and noting (2.7), we
have
\begin{equation*}
  \int_{\cal M} \nabla_k \varphi \langle X_i(u), \nabla_k u\rangle d\mu=0.
\end{equation*}
Since $X_i$ is Killing field,
\begin{equation*}
  \int_{\cal M} \langle \nabla_k (\varphi X_i(u)), \nabla_k u\rangle d\mu=0.
\end{equation*}
Since $N$ is a homogeneous space, due to the construction of $X_i$
by Helein, we can choose $\varphi_i$ so that
\begin{equation*}
  \sum_{i}\varphi_i X_i(u)
\end{equation*}
is any compactly supported vector field (along $u$). This implies
that $u$ is a weak harmonic map.

We define
 $$\Sigma =\bigcap_{R>0} \left\{x_0\in \Omega : B_R(x_0)\subset \cal M,\quad \liminf_{\varepsilon_k\to 0}
  \frac 1 {R^{n-2}}\int_{B_R(x_0)} |\nabla u_{\varepsilon} |^2 \,dx\geq \varepsilon_0\right
 \}
 $$
 for a sufficiently small constant $\varepsilon_0$. Then, $\cal
 H^{n-2}(\Sigma )<+\infty$.
 As in the proof of Theorem 1.1,  for any $x_0\notin \Sigma$ with $B_{R_0}(x_0)\subset \cal M \backslash \Sigma$,
for each $y\in B_{R_0/2}(x_0)$ and for each $\rho\in (0,R_0/2)$,
we have
 \begin{align}
 &\frac 1 {\rho^{n-2}}\int_{B_{\rho }(y)}|\nabla u|^2\,d\cal M\leq
 \lim_{\varepsilon_k\to 0}
\rho^{2-n}\int_{B_{ \rho}(y)} |\nabla u_{\varepsilon_k} |^2
\,d\cal M <2 \varepsilon_0
  \end{align}
 for a sufficiently small constant $\varepsilon_0>0$.
 It follows from the proof of  Bethuel in \cite{B} (see Lemma 3.3.14 of \cite {LW})
 that $u$ is smooth inside $\cal M\backslash \Sigma$.
\end{proof}

\section{Basic set up for the relaxed energy of Faddeev model}
The space $H_d$ mentioned in the introduction is not rigorously
defined since we have not made clear what
`$u$ approaches a constant value sufficiently fast at infinity' means. So the first task
of this section is to propose and justify a replacement for $H_d$.
We define
\begin{equation*}
    M=\{u:\Real^3\to S^2| \int_{\Real^3} \abs{\nabla u}^2+\abs{\nabla u}^4 dx< +\infty \}.
\end{equation*}
Since $\abs{\partial_k u\times \partial_l u}^2$ involves the
 $\nabla u$ up to the fourth power, $M$ is only a little smaller
than the set of finite energy maps. On the other hand, Theorem
\ref{thm:keylemma} below implies that the set of smooth maps which
are constant outside some large compact set is dense in $M$.
Remember whether this set is dense in the set of finite energy
maps is an open question and this is the main motivation of our
discussion of relaxed energy.
\begin{thm}\label{thm:keylemma}
    Let $u:\Real^3\to S^2$ be a map such that $E_F(u)$ and $\int_{\Real^3}\abs{\nabla u}^4 d\mu$ are finite. Then,
     there exists a sequence of $C^1$ maps $u_i:\Real^3\to S^2$ with $u_i$ constant outside some large ball
     (depending on $i$) such that $\nabla u_i$ converges to $\nabla u$ strongly in $L^4(\Real^3)\cap L^2(\Real^3)$.
\end{thm}

Before proving Theorem \ref{thm:keylemma}, we discuss some immediate consequences of it.
\begin{cor}\label{integerdegree}
    For each $u\in M$, $Q(u)$ is an integer.
\end{cor}
\begin{proof}
    Let $u_i$ be the sequence in Theorem \ref{thm:keylemma}.     $Q(u_i)$ is an integer since $u_i$ is of class $C^1$ and constant at infinity.
    \begin{equation*}
        Q(u_i)=\frac{1}{16\pi^2}\int_{\Real^3}\eta_i\wedge u^*_i \omega_{S^2}.
    \end{equation*}
    By the definition of $\eta_i$ (see equation (\ref{def:eta})) and the fact that $\nabla u_i\to \nabla u$ strongly
    in $L^4(\Real^3)\cap L^2(\Real^3)$ (which implies $u^*_i\omega_{S^2} \to u^*_i\omega_{S^2}$ strongly in $L^2(\Real^3)\cap L^1(\Real^3)$),
    $\eta_i$ converges to $\eta$ in $W^{1,p}$ for $1<p\leq 2$. In particular, $\eta_i$ converges strongly to $\eta$ in $L^2(\Real^3)$. Hence
    \begin{equation*}
        \lim_{i\to \infty}Q(u_i)=Q(u),
    \end{equation*}
    which implies $Q(u)\in \mathbb Z$.
\end{proof}

For any integer $d$, we define
\begin{equation*}
    M_d=\{u\in M| Q(u)=d\}.
\end{equation*}

As remarked earlier in this section, $M$ is only a little smaller than
the set of finite energy maps. We {\it may assume} $H_d\subset M_d$ and $H_d$ contains all smooth maps
which are constant outside some compact set. Another
consequence of Theorem \ref{thm:keylemma} is
\begin{equation}\label{eqn:weclaim}
        \inf_{u\in H_d} E_F(u)=\inf_{u\in M_d} E_F(u).
    \end{equation}
In fact,
    since $H_d\subset M_d$, we know
    \begin{equation*}
        \inf_{u\in H_d} E_F(u)\geq \inf_{u\in M_d} E_F(u).
    \end{equation*}
    On the other hand, for any $u\in M_d$, due to Theorem \ref{thm:keylemma}, there exists a sequence of $u_i\in H_d$ such that
    \begin{equation*}
        \begin{array}[]{c}
            E_F(u_i)\to E_F(u).
        \end{array}
    \end{equation*}
    Hence, the reverse inequality is also true. This proves our
    claim.

Moreover, according to the above, the definition of relaxed energy given in the introduction can be rewritten as
\begin{equation*}
    \tilde{E}_F(u)=\inf\left \{
    \liminf_{k\to \infty} E_F(u_k)\,| \quad \{u_k\} \subset M_d, u_k\rightharpoonup u \mbox{ weakly
    }
    \right\}.
\end{equation*}

The proof of Theorem \ref{thm:keylemma} depends on the following elliptic estimates.

Let $B_r$ be the ball with center at $0$ and radius $r$ in
$\Real^3$, let $u$ be a map from $B_1$ to $S^2$
such that
\begin{equation*}
    \int_{\partial B_1}\abs{\nabla u}^2 +R^{-2}\abs{\nabla u}^4 d\sigma=\epsilon,
\end{equation*}
for some positive $R$ (later $R$ will be large and $\epsilon$
small), and let $\xi$ be the average of $u$ over $\partial B_1$,
i.e.
\begin{equation*}
    \xi=\frac{1}{4\pi}\int_{\partial B_1}u d\sigma.
\end{equation*}
By the Poincar\'e inequality
 and using that $|u|=1$, we have
\begin{equation}\label{eqn:epsilon}
    1-\abs{\xi}\leq   1-\abs{\xi}^2 = \frac{1}{4\pi}\int_{\partial B_1}\abs{u-\xi}^2 d\sigma  \leq C \epsilon.
\end{equation}
As in \cite{LY1}, consider the harmonic function $V:B_2\setminus
B_1\to \Real^3$ such that
\begin{equation*}
    \triangle V=0, \quad \mbox{ in } B_2\setminus B_1,
\end{equation*}
and
\begin{equation*}
    V|_{\partial B_2}=\frac{\xi}{\abs{\xi}}, \quad V|_{\partial B_1}=u.
\end{equation*}
It follows from  (\ref{eqn:epsilon}) that
\begin{equation*}
    \norm{V-\xi}_{W^{1,4}(\partial B_2)}\leq C(1-\abs{\xi})\leq C\epsilon .
\end{equation*}
The average of $V-\xi$ on $\partial B_1$ is zero. By the Poincar\'e
inequality,
\begin{equation*}
    \norm{u-\xi}_{W^{1,4}(\partial B_1)}\leq C\norm{\nabla u}_{L^4(\partial B_1)}.
\end{equation*}

By the trace theorem of the Sobolev space, we   find
$w:B_2\setminus B_1\to \Real^3$ with the same boundary value as
$V-\xi$ on $\partial B_2\cup \partial B_1$ such that
\begin{equation*}
    \norm{w}_{W^{5/4,4}(B_2\setminus B_1)}\leq C\norm{V-\xi}_{W^{1,4}(\partial B_1\cup \partial B_2)}\leq C(\norm{\nabla u}_{L^4(\partial B_1)}+\epsilon ).
\end{equation*}
Then, $V-\xi-w$ has zero boundary value and
\begin{equation*}
    \triangle (V-\xi-w)=-\triangle w.
\end{equation*}
By the elliptic estimate, we have
\begin{equation*}
    \norm{V-\xi-w}_{W^{5/4,4}(B_2\setminus B_1)}\leq C (\norm{\nabla u}_{L^4(\partial B_1)}+\epsilon ).
\end{equation*}
Hence,
\begin{equation*}
    \norm{V-\xi}_{W^{5/4,4}(B_2\setminus B_1)}\leq C (\norm{\nabla u}_{L^4(\partial B_1))}+\epsilon ).
\end{equation*}

The following is a Sobolev's embedding theorem, which  is a
special case of Theorem 7.58 in Adams' book \cite{Adams}. For
reader's convenience, we quote it here.
\begin{thm}
    Let $s>0$, $1<p< q< +\infty$, and $\chi=s-\frac{n}{p}+\frac{n}{q}>0$. Then
    \begin{equation*}
        W^{s,p}(\Real^n)\to W^{\chi,q}(\Real^n).
    \end{equation*}
    The result holds true for domains with smooth boundary.
\end{thm}

Using the above theorem in the case
$s=5/4$, $p=4$, $\chi=1$ and
$q=6$,
\begin{equation*}
    \norm{V-\xi}_{W^{1,6}(B_2\setminus B_1)}\leq C (\norm{\nabla u}_{L^4(\partial B_1)}+\epsilon ).
\end{equation*}
In particular,
\begin{equation}\label{eqn:V6}
    \int_{B_2\setminus B_1}\abs{\nabla V}^6 dx\leq C (\int_{\partial B(1)} \abs{\nabla u}^4 d\sigma)^{3/2}
    +C\epsilon^6< C \epsilon^{3/2} R^3+C\epsilon^6<C \epsilon^{3/2}R^3.
\end{equation}
The last inequality holds because, in the form we need it,
$\varepsilon$ is small and  $R$ large. Applying the H\"older
inequality, we have
\begin{equation}\label{eqn:V4}
    \int_{B_2\setminus B_1} \abs{\nabla V}^4 dx\leq C \epsilon R^2.
\end{equation}
A similar argument works for the $L^2$ norm of $\nabla u$.
Therefore, we conclude
\begin{equation}
    \int_{B_2\setminus B_1} \abs{\nabla V}^2+R^{-2}\abs{\nabla V}^4 dx\leq C  \int_{\partial B_1}\abs{\nabla u}^2 +R^{-2}\abs{\nabla u}^4 d\sigma.
    \label{eqn:extension}
\end{equation}

With this estimate, we can now prove Theorem \ref{thm:keylemma}.

\begin{proof}[Proof of Theorem \ref{thm:keylemma}]
    Let $u$ be a map from $\Real^3\to S^2$ such that
    \begin{equation*}
        \int_{\Real^3} \abs{\nabla u}^2+\abs{\nabla u}^4 dx<C.
    \end{equation*}
    It suffices to show that for any $\varepsilon>0$ there is a $C^1$ map $w$ from $\Real^3$ to $S^2$ such that

    (1) $w$ is constant outside a large ball;

    (2)
    \begin{equation*}
        \int_{\Real^3} \abs{\nabla (u-w)}^2+\abs{\nabla (u-w)}^4 dx<\varepsilon.
    \end{equation*}

    Since
    \begin{equation*}
        \int_0^\infty \int_{\partial B_r} \abs{\nabla u}^2+\abs{\nabla u}^4 d\sigma dr< C,
    \end{equation*}
    for any $C_1>0$ small, there exists a sequence of $R_i$ going to infinity such that
    \begin{equation*}
        \int_{\partial B_{R_i}} \abs{\nabla u}^2 +\abs{\nabla u}^4 d\sigma <\frac{C_1}{R_i}.
    \end{equation*}
    Let $\xi_i$ be the average of $u$ over $\partial B_{R_i}$,
    \begin{equation*}
        \xi_i=\frac{1}{\abs{\partial B_{R_i}}}\int u d\sigma.
    \end{equation*}
    Then
    \begin{equation*}
        \frac{1}{\abs{B_{R_i}}}\int_{\partial B_{R_i}} \abs{u-\xi_i}^2\leq \int_{\partial B_{R_i}} \abs{\nabla u}^2 d\sigma\to 0.
    \end{equation*}

    Define $V$ to be the harmonic function defined on $B_{2R_i}\setminus B_{R_i}$ with boundary value
    \begin{equation*}
        V|_{\partial B_{R_i}}=u,\quad V|_{\partial B_{2R_i}}=\frac{\xi_i}{\abs{\xi_i}}.
    \end{equation*}
    Due to the estimate (\ref{eqn:extension}) and a scaling, we know
    \begin{equation*}
        \int_{B_{2R_i}\setminus B_{R_i}} \abs{\nabla V}^2+\abs{\nabla V}^4 dx\leq C R_i\int_{\partial B_{R_i}} \abs{\nabla u}^2 +\abs{\nabla u}^4 d\sigma < C_1.
    \end{equation*}

    For fixed $\varepsilon$, choose $C_1<\varepsilon/96$ and $R_i$ so large such that
    \begin{equation*}
        \int_{\Real^3\setminus B_{R_i}} \abs{\nabla u}^2+\abs{\nabla u}^4 dx<\varepsilon/6.
    \end{equation*}

    We then define $\tilde{w}$ to be

    (1) $u$ in $B_{R_i}$;

    (2) $\frac{V}{\abs{V}}$ in $B_{2R_i}\setminus B_{R_i}$;

    (3) $\frac{\xi_i}{\abs{\xi_i}}$ outside $B_{2R_i}$.

    Since, as it is proved in page 292 of \cite{LY1} (see also Lemma 9.7 of \cite{Hang}), the image of $V$ lies in $B_{3/2}\setminus B_{1/2}$
    for sufficiently large $i$, and
    the nearest-point-projection map restricted to $B_{3/2}\setminus B_{1/2}$ is Lipschitz and the Lipschitz constant is $2$,
    by replacing $\tilde{w}$ with its projection onto $S^2$, we conclude that $\tilde{w}$ has values in $S^2$,
    \begin{equation*}
      \int_{B_{2R_i}\setminus B_{R_i}} \abs{\nabla \tilde{w}}^2+\abs{\nabla \tilde{w}}^4
      dx<\frac{\varepsilon}{6}
    \end{equation*}
and
    \begin{equation*}
      \int_{\Real^3} \abs{\nabla (\tilde{w}-u)}^2 +\abs{\nabla (\tilde{w}-u)}^4
      dx<\varepsilon/3.
    \end{equation*}
    Finally, since we can modify $\tilde{w}$ to a smooth map $w$ with the same properties of $\tilde{w}$ in such a way that
    \begin{equation*}
      \int_{\Real^3} \abs{\nabla ({w}-u)}^2 +\abs{\nabla ({w}-u)}^4 dx<\varepsilon,
    \end{equation*}
we conclude the proof.

\end{proof}

\begin{rem}\label{rem:s3}
    Theorem \ref{thm:keylemma} holds true for $u:\Real^3\to S^3$ with finite energies $\int_{\Real^3}\abs{\nabla u}^4 dx$ and $\int_{\Real^3} \abs{\nabla u}^2dx$, the proof being the same.
\end{rem}

\section{Hopf lift and decomposition lemma}

In this section, we first prove a theorem about the Hopf lift. Our proof depends on Theorem \ref{thm:keylemma}
 in the previous section. Then, we use the Hopf lift to prove a key lemma, the cubic decomposition lemma, which will be essential in the proof of our main result.

Let us start from recalling some basic definitions of the Hopf
lift (see \cite{Hopf} for details). Let $\mathcal M$ be any
complete Riemannian manifold whose second cohomology group
$H^2(\mathcal M,\mathbb R)$ is trivial. Let $\Pi$ be the Hopf map
from $S^3\subset \Real^4$ to $S^2\subset \Real^3$ given by
\begin{equation*}
    \Pi(x_1,x_2,x_3,x_4)=\left(
    \begin{array}[]{c}
        2(x_1x_3+x_2x_4) \\
        2(x_2x_3-x_1x_4) \\
        x_1^2+x_2^2-x_3^2-x_4^2
    \end{array}
    \right).
\end{equation*}
Pulling back the volume form $\omega_{S^2}$ via $\Pi$ gives
\begin{equation*}
  \Pi^*\omega_{S^2}=4(dx_1dx_2+dx_3dx_4)=2d\alpha,
\end{equation*}
where $\alpha=x_1dx_2-x_2dx_1+x_3dx_4-x_4dx_3$. One can check by using Stokes theorem that
\begin{equation}\label{eqn:unitdegree}
    \frac{1}{16\pi^2}\int_{S^3} 2\alpha\wedge \Pi^*\omega_{S^2}=1.
\end{equation}

For any map $u:\mathcal M\to S^2$, a map $\tilde{u}$ from $\mathcal M$ to $S^3$ is called a Hopf lift of $u$ if
\begin{equation*}
  \Pi\circ \tilde{u}=u.
\end{equation*}
Any smooth 1 form $\eta$ on $\mathcal M$ satisfying
\begin{equation*}
  d\eta=u^*\omega_{S^2}
\end{equation*}
is called a gauge of $u$. For a fixed $u$, the gauge $\eta$ is not unique. However, if one fixes a Hopf lift $\tilde{u}$ of $u$, then $\tilde{u}$ determines a gauge for $u$ by
\begin{equation*}
  \eta=2u^*\alpha.
\end{equation*}
On the other hand, any gauge $\eta$ of $u$ determines a Hopf lift $\tilde{u}$ such that the above equation is true and such a map $\tilde{u}$ is unique up to multiplication by $e^{i\theta}$ for some constant $\theta\in \Real$ and
\begin{equation}\label{eqn:lift}
  \abs{\nabla \tilde{u}}^2=\frac{1}{4}\abs{\eta}^2+\abs{\nabla u}^2.
\end{equation}
See Lemma 2.1 of \cite{Hopf} for a proof.

We will generalize the above definition of Hopf lift and equation
(\ref{eqn:lift}) to Sobolev mapping $u$ with bounded Faddeev
energy and finite $\int_{\Real^3}\abs{\nabla u}^4 dx$. The proof
relies on Theorem \ref{thm:keylemma}. For such $u$, due to Theorem
\ref{thm:keylemma}, there exists a sequence of smooth maps $\{u_i\}$
such that $u_i$ is constant near infinity and $\nabla u_i$ converges to
$\nabla u$ strongly in both $L^2(\Real^3)$ and $L^4(\Real^3)$.

For each $u_i$, we define the Coulomb gauge $\eta_i$ by requiring
$\eta_i=\delta(-\frac{1}{4\pi\abs{x}}\star u_i^*\omega_{S^2})$.
Here $\star$ means convolution and $\delta$ is the formal adjoint
of $d$. One can check that
$d\eta_i=u_i^*\omega_{S^2}$ and $\delta \eta_i=0$.
Given $u_i$ and $\eta_i$, there
exists a lift $\bar{u}_i$ (called the Coulomb lift) such that
\begin{equation}\label{eqn:xx}
  \abs{\nabla \bar{u}_i}^2=\frac{1}{4}\abs{\eta_i}^2+\abs{\nabla u_i}^2.
\end{equation}
Moreover,
\begin{equation}\label{eqn:eta}
    \eta_i=\bar{u}_i^*(2\alpha)
\end{equation}
(see \cite{Hopf} for a proof).

Since $\nabla u_i$ strongly converge to $\nabla u$ in
$L^4(\Real^3)$, we know that $\eta_i$ converge to the Coulomb gauge
$\eta$ of $u$ strongly in $W^{1,2}(\Real^3)$. Due to the Sobolev
embedding, $\eta_i$ converge strongly in $L^2(\Real^3)$ and
$L^4(\Real^3)$. By  (\ref{eqn:xx}), we have
\begin{equation*}
  \int_{\Real^3} \abs{\nabla \bar{u}_i}^2 +\abs{\nabla \bar{u}_i}^4 dx<C.
\end{equation*}
Since the image of $\bar{u}_i$ is in $S^3$, $u_i\in W_{loc}^{1,4}(\Real^3)$. Hence, for each bounded domain $\Omega$,
we can take subsequence of $\nabla \bar{u}_i$ such that it converges weakly in $L^2$ and $L^4$ on $\Omega$. Moreover, due to equation (\ref{eqn:xx}) and its square,
we know the $L^2(\Omega)$ and $L^4(\Omega)$ norms of $\nabla \bar{u}_i$ also converge. Denote the limit by $\bar{u}$, then $\bar{u}_i$ converges to
$\bar{u}$ strongly in $W^{1,4}_{loc}(\Real^3)$. In particular,
\begin{equation*}
    \int_{\Real^3} \abs{\nabla \bar{u}}^4 dx<C.
\end{equation*}

 By the definition of Hopf lift and
(\ref{eqn:eta}),  we have for each $i$
\begin{equation*}
  \bar{u}_i^*(2\alpha \wedge \Pi^*\omega_{S^2}) = \eta_i\wedge u_i^* \omega_{S^2}.
\end{equation*}
For any bounded domain $\Omega$, we integrate the above identity
over $\Omega$ and take the limit $i\to \infty$ to get
\begin{equation*}
  \int_\Omega \bar{u}^*(2\alpha \wedge \Pi^*\omega_{S^2})=\int_\Omega \eta\wedge u^*\omega_{S^2}.
\end{equation*}

Since $\Omega$ is arbitrary, we have proved that
\begin{thm}\label{thm:lift}
  For each $u$ with finite Faddeev energy and finite $\int_{\Real^3} \abs{\nabla u}^4 dx$, there exists a Hopf lift $\bar{u}$ of $u$ such that
  \begin{equation*}
      \frac{1}{16\pi^2}\int_{\Real^3} \bar{u}^*(2\alpha\wedge \Pi^*\omega_{S^2})=\frac{1}{16\pi^2}\int_{\Real^3}\eta\wedge u^*\omega_{S^2}\in \mathbb Z.
  \end{equation*}
  Moreover,
  \begin{equation*}
    \int_{\Real^3} \abs{\nabla \bar{u}}^4 dx\leq C.
  \end{equation*}
\end{thm}

\begin{rem}\label{rem:integerdegree}
    The lift given in the above theorem is a special one. In our later proof, we will need the fact that for any map
    $\bar{u}:\Real^3\to S^3$ with finite $\int_{\Real^3}\abs{\nabla \bar{u}}^4 dx$,
    \begin{equation*}
        \frac{1}{16\pi^2}\int_{\Real^3} \bar{u}^*(2\alpha\wedge \Pi^*\omega_{S^2})
    \end{equation*}
    is an integer. One can give a proof of this by noting that the fact is true for a smooth map which is constant near infinity and Remark \ref{rem:s3}.
\end{rem}

We can now use Hopf lift to prove the cubic decomposition lemma.
It is an adapted version of Lemma 6.1 in Lin and Yang \cite{LY1}
(see also \cite {Hang}). One can see from the statement and the
proof of the following lemma the advantage of introducing the
space $M$. Notice that for $u\in M$ we only use the fact that $\int_{\Real^3}\abs{\nabla u}^4dx$ is finite, while estimates will depend only on Faddeev energy.

For the statement of the lemma we need some notations used in Section 6 of \cite{LY1}. Let $Q(R)$ be a cube of side
length $R>0$. Let $\{Q_i(R)\}_{i=1}^\infty$ be a cubic decomposition of $\Real^3$. That is,
\begin{equation*}
  \Real^3=\cup_{i=1}^\infty Q_i(R),
\end{equation*}
where the interiors of the cubes $Q_i(R)$ are mutually disjoint. For all $a\in \Real^3$, let
\begin{equation*}
  Q_i(R,a)=a+Q_i(R)
\end{equation*}
denote a translation of $Q_i(R)$ by $a$. We write
\begin{equation*}
  \Sigma_R=\cup_{i=1}^\infty \partial Q_i(R)
\end{equation*}
and
\begin{equation*}
  \Sigma_R(a)=a+\Sigma_R=\cup_{i=1}^\infty \partial Q_i(R,a)
\end{equation*}
for the unions of 2-dimensional faces of $Q_i(R)$'s and $Q_i(R,a)$'s.
    \begin{lem}\label{lem:cube}
      Suppose that $u\in M_d$ and $E_F(u)=C_1$. Let $\varepsilon$ be a small constant such that
      \begin{equation*}
          \int_{\Real^3}\varepsilon \abs{\nabla u}^4 dx\leq 1.
      \end{equation*}
      For any $\delta_0>0$, there are some constant $C$ depending only on $C_1$, a cubic decomposition $\{Q_i(R_0,a_0)\}$ for some large $R_0$ and a point $a_0\in \Real^3$ with $\abs{a_0}\leq R_0/4$ such that
      \begin{equation*}
          \int_{\Sigma_{R_0}(a_0)} (\abs{\nabla u}^2+\abs{\eta}^2+ \abs{\eta}^4) +\varepsilon \abs{\nabla u}^4 dx\leq \frac{C}{R_0} \leq \delta_0^2 << 1.
      \end{equation*}
     where $\eta$ is the Coulomb gauge of $u$. Moreover, for each $Q_i$, there is an integer $k_i$ such that
      \begin{eqnarray*}
        \abs{d-\sum_{i=1}^\infty k_i} &\leq & \sum_{i=1}^{\infty}\abs{\frac{1}{16\pi^2}\int_{Q_i}\eta(u)\wedge u^*\omega_{S^2}-k_i} \\
        &\leq& C \int_{\Sigma_{R_0}(a_0)} \abs{\nabla u}^2 +\abs{\eta}^4 +\varepsilon \abs{\nabla u}^4 d\sigma \\
        &\leq& \frac{C}{R_0}\leq \delta_0^2<<1.
      \end{eqnarray*}
    \end{lem}
    \begin{proof}
    Since $E_F(u)<C_1$,
    \begin{equation*}
        \int_{\Real^3} \abs{\nabla u}^2 +\abs{\eta}^2 +\abs{\eta}^4 +\varepsilon \abs{\nabla u}^4dx< C.
    \end{equation*}
    By Fubini's theorem,
    \begin{equation*}
        \int_0^{R/12}da_1 \int_0^{R/12}da_2 \int_0^{R/12}da_3\int_{\Sigma_R(a)} (\abs{\nabla u}^2+\abs{\eta}^2+\abs{\eta}^4+\varepsilon\abs{\nabla u}^4 d\sigma)\leq CR^2.
    \end{equation*}
    The mean value theorem then implies that there is a point $a_0\in \Real^3$ with $\abs{a_0}\leq R/4$ such that
    \begin{equation*}
        \int_{\Sigma_R(a_0)}(\abs{\nabla u}^2 +\abs{\eta}^2 +\abs{\eta}^4 +\varepsilon \abs{\nabla u}^4) d\sigma\leq \frac{C}{R}.
    \end{equation*}
    Here $C$ depends only on $C_1$. Therefore, we can choose $R$ sufficiently large so that the first statement of the lemma is true.
    Please note that $R$ depends only on $C_1$ and $\delta_0$.

Next, we need to estimate the error from $\int_{Q_i} \eta(u)\wedge
u^*\omega$ to an integer. For simplicity, we suppress lower index
$i$ and write $Q'$ for the cube whose center is the same as $Q$
and whose side length is twice as big as $Q$. Set
\begin{equation*}
    \int_{\partial Q} \abs{\nabla u}^2 +|\eta|^2+\abs{\eta}^4 +\varepsilon \abs{\nabla u}^4 d\sigma=\varepsilon_0.
\end{equation*}

Since $u$ has finite Faddeev energy and $L^4$ norm, we can use Theorem \ref{thm:lift}. There is a lift $\bar{u}$ such that
\begin{equation}\label{eqn:aa}
  \int_{Q} \eta\wedge u^*\omega_{S^2} =\int_Q \bar{u}^*(2\alpha\wedge \Pi^*\omega_{S^2}).
\end{equation}
We will now modify $\bar{u}$ outside $Q$.  Let $\xi$ be the average of $\bar{u}$ on $\partial Q$. Since
\begin{equation}\label{eqn:again}
    \abs{\nabla \bar{u}}^2=\abs{\nabla u}^2 +\frac{1}{4}\abs{\eta}^2,
\end{equation}
we know that
\begin{equation*}
  \int_{\partial Q} \abs{\nabla \bar{u}}^2\leq \varepsilon_0
\end{equation*}
and
\begin{equation*}
  \int_{\partial Q}\abs{\nabla \bar{u}}^4 \leq C(\varepsilon).
\end{equation*}
We can now define $V:Q'\setminus Q\to \Real^4$ to be the harmonic functions with boundary value
\begin{equation*}
  V|_{\partial Q}=\bar{u},\quad V|_{\partial Q'}=\frac{\xi}{\abs{\xi}}.
\end{equation*}
Set $w$ to a map from $\Real^3$ to $S^3$ by

(1) $w=\bar{u}$ in $Q$;

(2) $w=\frac{V}{\abs{V}}$ in $Q'\setminus Q$;

(3) $w=\frac{\xi}{\abs{\xi}}$ outside $Q'$.

Similar to the proof of estimate (\ref{eqn:V6}), we can show that
\begin{equation}\label{eqn:bb}
  \int_{Q'\setminus Q}\abs{\nabla w}^3 dx< C\varepsilon_0
\end{equation}
and ($w$ is constant outside $Q'$)
\begin{equation}\label{eqn:bc}
  \int_{\Real^3} \abs{\nabla w}^2+ \abs{\nabla w}^4 dx < +\infty.
\end{equation}
Due to equation (\ref{eqn:bc}) and Remark \ref{rem:integerdegree}, we know
\begin{equation}\label{eqn:cc}
    \frac{1}{16\pi^2}\int_{\Real^3} w^*(2\alpha\wedge \Pi^*\omega_{S^2})=k\in \mathbb Z.
\end{equation}
By  (\ref{eqn:bb}) and (\ref{eqn:cc}),
\begin{eqnarray*}
    \abs{\frac{1}{16\pi^2}\int_Q\eta\wedge u^*\omega_{S^2}-k}&=& \abs{\frac{1}{16\pi^2}\int_{\Real^3\setminus Q}w^*(2\alpha\wedge \Pi^*\omega_{S^2})} \\
    &\leq& C\varepsilon_0.
\end{eqnarray*}
\end{proof}

As an application of Lemma \ref{lem:cube}, we show that the weak limit of a sequence of maps in $M_d$ has also an integer degree.
Notice that such a weak limit may not be in $M$.
\begin{thm} \label{thm:degree}
    Let $u_j$ be a sequence in $M_d$ with bounded $E_F(u_j)$. Assume that $u_j$ converges weakly to $u$
    in the sense that $u_j$ converges strongly in $L^2$ to $u$ and $\nabla u_j$ converges weakly in $L^2$ to $\nabla u$ and $\partial_k u_j\times \partial_l u_j$ converges weakly in $L^2$ to $\partial_k u\times \partial_l u$. Then
    \begin{equation*}
        \frac{1}{16\pi^2}\int_{\Real^3} \eta\wedge u^*\omega_{S^2} \in \mathbb Z
    \end{equation*}
    where $\eta=\delta (-\frac{1}{4\pi\abs{x}}\star u^*\omega_{S^2})$.
\end{thm}
\begin{proof}
    For any small positive number $\varepsilon_0$, we will prove that the difference between
    \begin{equation*}
        \frac{1}{16\pi^2}\int_{\Real^3} \eta\wedge u^*\omega_{S^2}
    \end{equation*}
    with some integer is smaller than $\varepsilon_0$.

    For this $\varepsilon_0$, we can use Lemma \ref{lem:cube} for each $u_j$ to obtain a cubic decomposition for each $u_j$.
    Precisely, there are $a_j\in \Real^3$ and $\abs{a_j}\leq \frac{R_0}{4}$ such that $\Sigma_{R_0}(a_j)$ decompose the space into
    cubes for which we write $Q_{j,i}$. By choosing a subsequence, we conclude that $a_j$ converges to $a$ and each $Q_{j,i}$
     converges to $Q_i$. Let $\mathcal{E}_{j,i}$ denote the difference between
    \begin{equation*}
        \int_{Q_{j,i}} \frac{1}{16\pi^2} \eta_j\wedge u^*_j\omega_{S^2}
    \end{equation*}
    and the nearest integer. We know from Lemma \ref{lem:cube}
    \begin{equation*}
        \sum_{i=1}^\infty \mathcal{E}_{j,i}<\varepsilon_0.
    \end{equation*}

    Now we claim:
    \begin{equation*}
        \int_{Q_{j,i}} \eta_j\wedge u_j^*\omega_{S^2} \to \int_{Q_i} \eta\wedge u^*\omega_{S^2}.
    \end{equation*}
    We know that for each fixed compact domain $\Omega$
    \begin{equation*}
        \int_{\Omega} \eta_j\wedge u_j^*\omega_{S^2} \to \int_{\Omega} \eta\wedge u^*\omega_{S^2}.
    \end{equation*}
    Therefore, to prove the claim, it suffices to show that if $\abs{\Omega}\to 0$, then
    \begin{equation*}
        \int_{\Omega} \eta_j\wedge u_j^*\omega_{S^2}\to 0
    \end{equation*}
    uniformly. This is true due to H\"older inequality and the fact that the $L^2$ norm of $u_j^*\omega$ and $L^4$ norm of $\eta_j$ are bounded uniformly.

    Given the claim, we know that the difference between
    \begin{equation*}
        \frac{1}{16\pi^2}\int_{Q_i} \eta\wedge u^*\omega_{S^2}
    \end{equation*}
    and the nearest integer is no bigger than
    \begin{equation*}
        \liminf_{j\to \infty} \mathcal{E}_{j,i}.
    \end{equation*}
    Since $\mathcal{E}_{j,i}$ are all positive, it is elementary to prove
    \begin{equation*}
        \sum_{i=1}^\infty \liminf_{j\to \infty} \mathcal{E}_{j,i}\leq \varepsilon_0.
    \end{equation*}
    This justifies the statement at the beginning of this proof.
\end{proof}

\section{Minimizing relaxed energy via perturbed functional}
For each $\varepsilon >0$, we define perturbed energy as follows
\begin{equation*}
    E_{F, \varepsilon }(u)=\int_{\Real^3} \abs{\nabla u}^2 +\frac{1}{2}\sum_{1\leq k<j\leq 3}\abs{\partial_k u\times \partial_l u}^2 +\varepsilon\abs{\nabla u}^4 dx.
\end{equation*}

\begin{thm}
    There is an infinite subset $S$ of $\mathbb Z$ such that for each $d\in S$ the minimizing problem of $E_{F,\varepsilon}$ in $X_d$ has a solution.
\end{thm}

The proof follows the same argument as in Lin and Yang \cite{LY1}.
The first ingredient is an energy growth lemma.

\begin{lem}
    \label{lem:energygrowth}
    There is a universal constant $C$ such that
    \begin{equation*}
        E_{F,\varepsilon,d}=\inf \{E_{F,\varepsilon}(u)| u\in M_d\}\leq C\abs{d}^{3/4}.
    \end{equation*}
\end{lem}

The lemma follows from the same proof as Lemma 5.1 in \cite{LY1}, since the test map constructed in their proof lies in $M_d$.

The second ingredient is an energy splitting inequality.

\begin{lem}
    \label{lem:split}
    If $d\notin S$, then we can find $l(>1)$ integers $d_1,\cdots,d_l\in S$, whose absolute values are bounded by $\abs{d}$, such that
    \begin{equation*}
        d=d_1+\cdots+ d_l
    \end{equation*}
    and
    \begin{equation*}
        E_{F,\varepsilon,d}\geq E_{F,\varepsilon,d_1}+\cdots +E_{F,\varepsilon,d_l}.
    \end{equation*}
\end{lem}

For any $d\notin S$, we consider a minimizing sequence $u_i$ of
$E_{F,\varepsilon}$ in $M_d$. Then
\begin{equation*}
    E_F(u_i)\leq E_{F,\varepsilon}(u_i)\leq C.
\end{equation*}

By Theorem \ref{thm:keylemma}, there are integers $k_{i,j}$ and cubes $Q_j(R_0,a_i)$ such that
\begin{eqnarray}\label{eqn:620}
    &&\sum_{j=1}^\infty \abs{\frac{1}{16\pi^2}\int_{Q_j(R_0,a_i)}\eta_i\wedge u_i^*\omega_{S^2}}\\
    \nonumber
    &\geq&\sum_{j=1}^\infty \abs{k_{i,j}}-\sum_{j=1}^\infty \abs{\frac{1}{16\pi^2}\int_{Q_j(R_0,a_i)}\eta_i\wedge u^*_i\omega_{S^2}-k_{i,j}} \\
    \nonumber
    &\geq& m_i-\delta_0^2.
\end{eqnarray}
Here $m_i$ is the number of nonzero $k_{i,j}$'s and $\delta_0$ is some very small positive number.

The left hand side of equation (\ref{eqn:620}) is bounded above by the energy upper bound $C$. Therefore, by taking subsequence,
we may assume $m_i=m$ for all $i$. For convenience, we write $Q_i^1,\cdots,Q_i^m$ for $Q_j(R_0,a_i)$ (with nonzero $k_{i,j}$).

The proof of Lemma  \ref{lem:split} is the same as that  of
Theorem 6.2 in \cite{LY1}. For the completeness, we repeat it
here. We introduce the subsets, $S_1$, $S_2$, \ldots, $S_l$ of the
set $S=\{1\cdots m\}$ inductively. Let $s_1=1$. We choose a
subsequence of $u_i$, still denoted by $u_i$, such that there is a
subset $S_1$ of $S$ with $s_1\in S_1$ and
\begin{equation*}
    \lim_{i\to \infty} \mbox{dist}(Q_i^{s_1},Q_i^k)
\end{equation*}
exists (finite) for each $k\in S_1$. Moreover, for any $k\in
S\setminus S_1$, we have
\begin{equation*}
    \lim_{i\to \infty} \mbox{dist}(Q_i^{s_1}-Q_i^k)=\infty.
\end{equation*}
The existence of such a subsequence of $u_i$ can be checked by
induction.

If $S\setminus S_1$ is nonempty, we set $s_2=\min{k\in S\setminus
S_1}$ and repeat the above procedure to find a subsequence of
$u_i$ such that there is a subset $S_2$ of $S\setminus S_1$ so
that
\begin{equation*}
    \lim_{i\to \infty} \mbox{dist}(Q_i^{s_2},Q_i^k)
\end{equation*}
exists (finite) for each $k\in S_2$ and for any $k\in S\setminus
(S_1\cup S_2)$, we have
\begin{equation*}
    \lim_{i\to \infty} \mbox{dist} (Q_i^{s_2},Q_i^k)=\infty.
\end{equation*}

Continuing this way we can find $l$ numbers
$1=s_1<s_2<\cdots<s_l$, the corresponding subsets $S_1,\cdots,S_l$
of $S$ and a subsequence of $u_i$ (denoted again $u_i$ for
simplicity) with the following properties:
\begin{equation*}
    \begin{array}[]{rcl}
        S&=& \{1,2,\cdots,m\}=S_1\cup S_2\cup\cdots\cup S_l, \\
        S_s\cap S_t&=& \emptyset \mbox{ for } s\ne t, s,t=1,2,\cdots,l; \\
        s_1&=& 1,\\
        S_1&=& \{k\in S| \lim_{i\to \infty}\mbox{dist}(Q_i^{s_1},Q_i^k)<\infty\};\\
        s_2&\in& S\setminus S_1,\\
        S_2&=& \{k\in S| \lim_{i\to \infty}\mbox{dist}(Q_i^{s_2},Q_i^k)<\infty\};\\
        \cdots&&\cdots\cdots\\
        s_l&\in& S\setminus \cup_{s=1}^{l-1} S_s,\\
        S_l&=& \{k\in S| \lim_{i\to \infty}\mbox{dist}(Q_i^{s_l},Q_i^k)<\infty\}.\\
    \end{array}
\end{equation*}
Let $x_{i,1},x_{i,2},\cdots,x_{i,l}$ be the centers of cubes
$Q_i^{s_1},\cdots,Q_i^{s_l}$, denoted thereafter as
$Q_{i,1},Q_{i,2},\cdots,Q_{i,l}$.  Let
$v_{i,s}(x)=u_i(x-x_{i,s}),s=1,2,\cdots,l$ and set
\begin{equation*}
    R_s=\max \{\lim_{i\to \infty} \mbox{dist} (Q_{i,s}, Q_i^t)|t \in S_s\}, \quad s=1,2,\cdots,l.
\end{equation*}
Then $\abs{x_{i,s}-x_{i,t}}\to \infty$ as $i\to \infty$ for $s\ne
t$ and $s,t=1,2,\cdots,l$. By further taking subsequences if
necessary, we see that for $s=1,2,\cdots,l$
\begin{equation*}
    d_s\equiv \lim_{i\to \infty}\sum_{t\in S_s} k(Q_{i}^t)
\end{equation*}
exists and that $v_{i,s}$ converges to $v_s$ weakly  as $i\to
\infty$.

By Lemma \ref{lem:cube}, we obtain
\begin{equation*}
    \sum_{s=1}^l\sum_{t\in S_s} k(Q_i^t)=d
\end{equation*}
for each $i$. Thus $d=\sum_{s=1}^l d_s$ follows.

Due to the lower semicontinuity of $E_{F,\varepsilon}$, we know
\begin{equation*}
    E_{F,\varepsilon}(v_s)<C,\quad s=1,2,\cdots,l.
\end{equation*}
By Theorem \ref{thm:degree}, the Hopf degree of $v_s$ is an integer.
Following Lemma 6.3 in Lin and Yang \cite{LY1}, we conclude that $Q(v_s)=d_s$.

For any large but fixed $R$, when $i$ is sufficiently large,
$B_R(x_{i,s})$ are disjoint for different $s$. Therefore
\begin{equation*}
    E_{F,\varepsilon}(u_i)\geq \sum_{s=1}^{l} E_{F,\varepsilon}(u_i, B_R(x_{i,s})),
\end{equation*}
where $E_{F,\varepsilon}(u,\Omega)$ means the perturbed energy of
$u$ integrating over $\Omega$. Taking $i\to \infty$,
\begin{equation*}
    E_{F,\varepsilon ;d}\geq \sum_{s=1}^l E_{F,\varepsilon}(v_l,B_R(0)).
\end{equation*}
Since $R$ is an arbitrary positive number, we have
\begin{equation*}
    E_{F,\varepsilon ;d}\geq \sum_{s=1}^l E_{F,\varepsilon ;d_s}.
\end{equation*}
If $d_i$ is not in $S$, we can repeat the above argument until we
arrive at a decomposition in which each $d_i$ lies in $S$.

Combining Lemma \ref{lem:energygrowth}, Lemma
\ref{lem:split} and the fact that for all $d\ne 0$,
$E_{F,\varepsilon;d}$ have a uniform lower bound, one can see that
$S$ must be an infinite set.

We have another existence result.
\begin{thm}\label{thm:one}
    For $d=\pm 1$, the minimizing problem of $E_{F,\varepsilon}$ on $M_d$ has a solution when $\varepsilon$ is very small.
\end{thm}

The proof relies on Lemma \ref{lem:split}, an upper bound estimate of $E_{F,\varepsilon;1}$ and a lower bound estimate of $E_{F,\varepsilon;d}$ for all $d$.

As for the lower bound, it was proved by Vakulenko and Kapitanski \cite{VK} and improved by Lin and Yang \cite{LY2} that
\begin{equation}
    E_{F,\varepsilon}(u)\geq E_F(u)\geq 3^{3/8}8\sqrt{2}\pi^2 \abs{\deg u}^{3/4}.
    \label{eqn:lower}
\end{equation}
The upper bound was proved by Ward \cite{Ward} and Lin and Yang \cite{LY2} for Faddeev energy
\begin{equation}\label{eqn:upper}
    E_{F,0;1}\leq 32\sqrt{2}\pi^2.
\end{equation}
\begin{rem}\label{rem:upper}
    The proof in \cite{Ward} and \cite{LY2} was indirect. In fact, as pointed out by the authors,
     the test map $\Phi$ is nothing but a Hopf map from ball of radius $\frac{1}{\sqrt{2}}$ composed with a stereographic projection.
     We may calculate directly that the Faddeev energy of $\Phi$ is $32\sqrt{2}\pi^2$, which gives the upper bound.
     Moreover, by direct computation, the $L^4$ norm of gradient of this map is bounded. We list below some intermediate results of this calculation.
\end{rem}

The test map $\Phi:\Real^3\to S^2\subset \Real^3$ is given by (see (4.13) in \cite{LY2})
\begin{equation*}
    \Phi(x,y,z)=\left(
    \begin{array}[]{c}
        \frac{4a}{(r^2+a^2)^2}(2axz+(r^2-a^2)y) \\
        \frac{4a}{(r^2+a^2)^2}(2ayz-(r^2-a^2)x) \\
        1-\frac{8a^2}{(r^2+a^2)^2}(z^2-r^2)
    \end{array}
    \right),
\end{equation*}
where $a=1/\sqrt{2}$ and $r^2=x^2+y^2+z^2$.

One can calculate
\begin{equation*}
    \abs{\nabla \Phi}^2=\frac{64}{(1+4 r^2)^2}.
\end{equation*}
Integrating over $\Real^3$ ,
\begin{equation*}
    \int_{\Real^3} \abs{\nabla \Phi}^2 dx=16\sqrt{2}\pi^2.
\end{equation*}
Moreover, we can find out that
\begin{equation*}
    \int_{\Real^3} \abs{\nabla \Phi}^4 dx=128\sqrt{2}\pi^2.
\end{equation*}
It is important for us that this number is finite.
For the other part of the energy,
\begin{equation*}
    \abs{\partial_x \Phi\times \partial_y \Phi}^2=\frac{1024(1-2x^2-2y^2+2z^2)^2}{(1+2r^2)^6}
\end{equation*}
and
\begin{equation*}
    \sum_{1\leq k<l\leq 3} \abs{\partial_k \Phi\times \partial_l \Phi}^2=\frac{1024}{(1+2r^2)^4}.
\end{equation*}
Integrating over $\Real^3$,
\begin{equation*}
    \int_{\Real^3} \sum_{1\leq k<l\leq 3} \abs{\partial_k \Phi\times \partial_l \Phi}^2 dx=32\sqrt{2}\pi^2.
\end{equation*}
In summary,
\begin{equation*}
    E_F(\Phi)=16\sqrt{2}\pi^2+\frac{1}{2}32\sqrt{2}\pi^2=32\sqrt{2}\pi^2.
\end{equation*}

Therefore, when $\varepsilon$ is very small, we may assume
\begin{equation*}
    E_{F,\varepsilon,\pm 1}\leq 32\sqrt{2}\pi^2+0.01.
\end{equation*}

For such $\varepsilon$, we can show that $\pm 1\in S$. In fact, if otherwise,
we would have a decomposition of $d=1$. If it is $1=1+(-2)$, then the energy inequality in
Lemma \ref{lem:split} says
\begin{equation}\label{eqn:energy}
    32\sqrt{2}\pi^2+0.01\geq 3^{3/8}8\sqrt{2}\pi^2 (1+2^{3/4}),
\end{equation}
which is not possible by direct computation. There are other decompositions, such as $1=1+1+(-1)$, $1=3+(-2)$ and so on.
It is easy to see that the corresponding right hand sides of (\ref{eqn:energy}) are even larger
so that the inequality is not possible for all decomposition. Therefore, $1\in S$ and so is $-1$.

We now prove Theorem 1.4.

\begin{proof}[Proof of Theorem 1.4]
    Let $\varepsilon_j$ be a sequence of positive number going to zero. Let $u_j$ be the minimizer of $E_{F,\varepsilon_i}$ in $M_1$.

    We can now apply Lemma \ref{lem:cube} to $u_j$. The same argument as in Lemma \ref{lem:split} works also for $u_j$. If $l=1$ which means there is no splitting at all, we get a limiting map $u$. $E_F(u)$ is finite due to the weakly lower semicontinuity of $E_F$. By Theorem \ref{thm:degree}, we know that the Hopf degree of $u$ is an integer. Since there is no splitting, we infer from the construction of $u$ that $Q(u)$ is $\delta_0$ close to $1$. Therefore, $u\in X_1$ and by the definition of relaxed energy
    \begin{equation}\label{eqn:one}
        \lim_{j\to \infty} E_F(u_j)\geq \tilde{E}_F(u)\geq  \inf_{w\in X_1} \tilde{E}_F(w).
    \end{equation}
    On the other hand, for any positive $\delta>0$, we can find $w\in X_1$ such that
    \begin{equation*}
        \tilde{E}_F(w)\leq \inf_{w\in X_1}\tilde{E}_F(w) +\delta.
    \end{equation*}
    By the definition of relaxed energy, we can find $w'\in M_1$ such that
    \begin{equation*}
        E_F(w')\leq \inf_{w\in X_1}\tilde{E}_F(w) +\delta.
    \end{equation*}
    For this fixed $w'$, $\int_{\Real^3}\abs{\nabla w'}^4 dx$ is finite. therefore, when $\varepsilon_j$ is sufficiently small,
    \begin{equation}\label{eqn:two}
        E_F(u_j)\leq E_F(w')+\delta \leq \inf_{w\in X_1}\tilde{E}_F(w)+3\delta.
    \end{equation}
    By equation (\ref{eqn:one}) and (\ref{eqn:two}), $u$ achieves the minimum of $\tilde{E}_F$ in $X_1$.

    Next, we will prove that the case $l>1$ is not possible. If otherwise, there exists $d_1,\cdots,d_l\in \mathbb Z$ and $x_{j,1},\cdots,x_{j,l}\in \Real^3$ such that
    \begin{equation*}
      1=d_1+\cdots+d_l
    \end{equation*}
    and
    $u_j(x-x_{j,s})$ converges weakly to $v_s$. These $v_s$ has finite energy and by Theorem \ref{thm:degree} the Hopf degree of $v_s$ is the integer $d_s$. For any fixed $R>0$, when $j$ is sufficently large so that $B(R,x_{j,s})$ are disjoint for different $s$,
    \begin{equation*}
      E_F(u_j)\geq \sum_{s=1}^l E_F(u_j,B_R(x_{j,s})).
    \end{equation*}
    Let $j\to \infty$,
    \begin{equation*}
      \inf_{w\in X_1}\tilde{E}_F(w)\geq \sum_{s=1}^l E_F(v_s,B_R(0)).
    \end{equation*}
    Since $R$ is arbitrary,
    \begin{equation*}
      \inf_{w\in X_1}\tilde{E}_F(w)\geq \sum_{s=1}^l E_F(v_s).
    \end{equation*}
    Since $v_s$ has bounded energy and $Q(v_s)=d_s$, the energy lower bound (\ref{eqn:lower}) is valid for $E_F(v_s)$.
    Moreover, because $\Phi$ in Remark \ref{rem:upper} is in $M_1$, the same test function implies upper bound for relaxed energy
    \begin{equation*}
      \inf_{w\in X_1}\tilde{E}_F(w)\leq 32\sqrt{2}\pi^2.
    \end{equation*}
    For the same reason as in Theorem \ref{thm:one}, the splitting is  not possible.

   Due to (\ref{eqn:weclaim}) and (5.5), we have
    \begin{equation*}
        \inf_{u\in H_1} E_F(u)\geq \min_{u\in X_1} \tilde{E}_F(u).
    \end{equation*}
    (1.6) follows from the definition of the relaxed energy $\tilde{E}_F$.
\end{proof}

\begin{acknowledgement}
 {The research  of the second author and the third author was supported by the Australian Research Council
grant DP0985624. Partial work was done when the first author
visited the University of Queensland in August 2009.}
\end{acknowledgement}

\end{document}